\documentclass[]{amsart}

\usepackage{amssymb}
\usepackage{enumerate}

\usepackage[all]{xy}

\newtheorem{proposition}{Proposition}
\newtheorem{theorem}[proposition]{Theorem}

\newtheorem{corollary}[proposition]{Corollary}

\theoremstyle{definition}  
\newtheorem{definition}[proposition]{Definition}
\newtheorem{example}[proposition]{Example}

\newcommand{\thmref}[1]{Theorem~\ref{#1}}
\newcommand{\propref}[1]{Proposition~\ref{#1}}

\newcommand{\cororef}[1]{Corollary~\ref{#1}}

\newcommand{\examref}[1]{Example~\ref{#1}}

\newcommand{\secat}{{\rm secat}\,}
\newcommand{\cat}{{\rm cat}\,}
\newcommand{\Relcat}{{\rm Relcat}\,}
\newcommand{\relcat}{{\rm relcat}\,}
\newcommand{\relcatk}{{\rm relcat}_k\,}
\newcommand{\relcatun}{{\rm relcat}_1\,}

\newcommand{\hcat}{{\rm hcat}}
\newcommand{\Hcat}{{\rm Hcat}}

\newcommand{\Pushcat}{{\rm Pushcat}\,}

\newcommand{\ipu}{{i+1}}
\newcommand{\nmu}{{n-1}}
\newcommand{\imu}{{i-1}}

\newcommand{\id}{{\rm id}}

\newcommand{\pro}{{\rm pr}}

\renewcommand{\leq}{\leqslant}
\renewcommand{\geq}{\geqslant}

\title{Yet another Hopf invariant}

\author{Jean-Paul Doeraene and Mohammed El Haouari}

\subjclass[2010]{55M30}  

\keywords{Ganea fibration, sectional category, Hopf invariant.} 

\begin{document}

\maketitle

\begin{abstract}
The classical Hopf invariant is defined for a map $f\colon S^r \to X$. Here we define `hcat' which is some kind of Hopf invariant built with a construction in Ganea's style, valid for maps not only on spheres but more generally on a `relative suspension' $f\colon \Sigma_A W \to X$.
We study the relation between this invariant and the sectional category and the relative category of a map. In particular, for $\iota_X\colon A\to X$ being the `restriction' of $f$ on $A$, we have $\relcat \iota_X \leq \hcat f \leq \relcat \iota_X +1$ and $\relcat f \leq \hcat f$.
\end{abstract}

Our aim here is to make clearer the link between the Lusternik-Schnirelmann category (cat), more generally the  `relative category' (relcat), closely related to James'sectional category (secat), and the Hopf invariants. In order to do this, we introduce a new integer, namely $\hcat$, that combines the Iwaze's version of Hopf invariant \cite{Iwa98}, based on the {\em difference up to homotopy between two maps} defined for a given section of a Ganea fibration, and the framework of the sectional and relative categories, searching for the {\em least integer} such that the Ganea fibration has a section, possibly with additional conditions. 
To do this combination, we simply define our invariant $\hcat$, as the least integer such that the Ganea fibration has a section $\sigma$ with additional condition that the corresponding two maps ($f\circ\sigma$ and $\omega_n$ in this paper) are homotopic.

It appears that for $f\colon S^r \to X$ or even for $f\colon \Sigma W \to X$,  we obtain an integer that can be either $\cat(X)$, or $\cat(X)+1$. More generally, for any $f\colon \Sigma_A W \to X$, we have $\relcat(f\circ\theta) \leq \hcat(f) \leq \relcat(f\circ\theta)+1$, where $\theta\colon A \to \Sigma_A W$ is the map arising in the construction of $\Sigma_A W$.

In section \ref{cofibre}, we study the influence of $\hcat$ in a homotopy pushout.
In section \ref{strong}, we introduce the `strong' version of our invariant, 
and we obtain another important inequality: for any $f\colon \Sigma_A W \to X$, we have $\relcat(f) \leq \hcat(f)$.
In section \ref{sigmabars}, we give alternative equivalent conditions to get $\hcat$.
Applications and examples are given.

\section{The Hopf category}

We work in the category of pointed topological spaces. All constructions are made up to homotopy. A `homotopy commutative diagram' has to be understood in the sense of \cite{Mat76}.

\smallskip
Recall the following construction:
\begin{definition}\label{ganea}
For any map $\iota_X\colon A \to X$,
the \emph{Ganea construction} of $\iota_X$
is the following sequence of homotopy commutative diagrams ($i \geq 0$):
$$\xymatrix{
&A\ar[dr]_{\alpha_{i+1}}\ar@/^/[rrrd]^{\iota_X}\\
F_i\ar[rd]_{\beta_i}\ar[ur]^{\eta_{i}}&&G_{i+1}\ar[rr]|-(.35){g_{i+1}}&&X\\
&G_i\ar[ru]^{\gamma_{i}}\ar@/_/[rrru]_{g_{i}}}$$
where the outside square is a homotopy pullback, the inside
square is a homotopy pushout and the map $g_{i+1} = (g_i,\iota_X) \colon G_{i+1} \to X$
is the whisker map induced by this homotopy pushout.
The iteration starts with  $g_0 = \iota_X \colon A \to X$.
We set $\alpha_0 = \id_A$.
\end{definition}

For any $i \geq 0$, there is a whisker map $\theta_i = (\id_A, \alpha_i)\colon A \to F_i$ induced by the homotopy pullback. Thus we have the sequence of maps $\xymatrix{A\ar[r]|(.4){\theta_i}&F_i\ar[r]|(.6){\eta_i}&A}$ and $\theta_i$ is a homotopy section of  $\eta_i$. Moreover we have $\gamma_i \circ \alpha_i \simeq \alpha_{i+1}$, thus also $\alpha_{i+1}\simeq \gamma_i \circ \gamma_{i-1} \circ \dots \circ \gamma_0$.

We denote by $\gamma_{i,j}\colon G_i \to G_j$ the composite $\gamma_{j-1} \circ \dots \circ \gamma_{i+1} \circ \gamma_i$ (for $i < j$) and set $\gamma_{i,i} = \id_{G_i}$.

Of course, everything in the Ganea construction depends on $\iota_X$. We sometimes denote $G_i$ by $G_i(\iota_X)$ to avoid ambiguity.

\begin{definition}\label{LSganea}
Let $\iota_X\colon A \to X$ be any map. 

1) The \emph{sectional category} of $\iota_X$ is the least integer $n$ such that the map $g_n\colon G_n(\iota_X)\to X$ has a homotopy section, i.e. there exists a map $\sigma\colon X \to G_n(\iota_X)$ such that $g_n \circ \sigma \simeq \id_X$.

2) The \emph{relative category} of $\iota_X$ is the least integer $n$ such that the map $g_n\colon G_n(\iota_X)\to X$ has a homotopy section $\sigma$ and $\sigma \circ \iota_X \simeq \alpha_n$.

3) The  \emph{relative category of order $k$} of $\iota_X$  is the least integer $n$ such that the map $g_n\colon G_n(\iota_X)\to X$ has a homotopy section $\sigma$ and $\sigma \circ g_k \simeq \gamma_{k,n}$.
\end{definition}

We denote the sectional category  by $\secat(\iota_X)$, the relative category by $\relcat(\iota_X)$, and the relative category of order $k$ by   $\relcatk(\iota_X)$.
If $A = \ast$,  $\secat(\iota_X) = \relcat(\iota_X)$ and is denoted simply by $\cat(X)$; this is the `normalized' version of the Lusternik-Schnirelmann category.

Clearly, $\secat(\iota_X) \leq \relcat(\iota_X)$. We have also  $\relcat(\iota_X)\leq \relcatun(\iota_X)$, see \propref{majorationhcat} below.

\smallskip
In the sequel, we will consider a given homotopy pushout:
$$\xymatrix{
W\ar[r]^\eta\ar[d]_\beta&A\ar[d]^\theta\\
A\ar[r]_(.4)\theta&\Sigma_A W
}$$
In other words, the map $\theta$ is a map such that $\Pushcat \theta \le 1$ in the sense of \cite{DoeHa13}.  We call this homotopy pushout a `relative suspension' because in some sense, $A$ plays the role of the point in the ordinary suspension.

We also consider any map $f\colon \Sigma_A W \to X$, and set $\iota_X = f\circ \theta$. 

We don't assume $\eta \simeq \beta$ in general. This is true, however, if $\theta$ is a homotopy monomorphism, and in this case we can  `think' of $\iota_X$ as the `restriction' of $f$ on $A$.

\smallskip
For $n \geq 1$, consider the following homotopy commutative diagram:
\[\begin{split}\xymatrix@R=1pc@C=1pc{
&W\ar[rr]^\beta\ar[ld]_\eta\ar@{-->}[dd]|\hole&&A\ar[ld]_(0.6)\theta\ar[dd]^(.7){\alpha_{n-1}}|\hole\ar[rd]^{\theta}\\
A\ar[rr]^(.65){\theta}\ar@{=}[dd]&&\Sigma_A W\ar@{.>}[dd]^(.3){\omega_n}\ar@{=}[rr]&&{\Sigma_A W}\ar[dd]^f\\
&F_{n-1}(\iota_X)\ar[dl]\ar[rr]|!{[ur];[dr]}\hole && G_{n-1}(\iota_X)\ar[dl]^(.4){\gamma_{n-1}}\ar[dr]^(.55){g_{n-1}}\\
A\ar[rr]_{\alpha_n}&&G_n(\iota_X)\ar[rr]_{g_n}&&X
}\end{split}\label{diag1}\tag{\dag}\]
where the map $W\to F_{n-1}$ is induced by the bottom outer homotopy pullback and the map $\omega_n\colon \Sigma_A W\to G_n$ is induced by the top inner homotopy pushout.
We have $f \simeq g_n\circ \omega_n$ by the `Whiskers maps inside a cube' lemma (see \cite{DoeHa13}, Lemma 49). Also notice that $\alpha_n \simeq \omega_n \circ \theta \simeq \gamma_\nmu \circ \alpha_\nmu $;
so $\omega_n \simeq (\alpha_n, \alpha_n)$ is the whisker map of two copies of $\alpha_n$ induced by the homotopy pushout $\Sigma_AW$.
Finally, for all $k \geq 1$, we can see that $\omega_n \simeq \gamma_{k,n} \circ \omega_k$.

\begin{definition} The {\em Hopf category} of $f$ is the least integer $n \geq 1$ such that  $g_n \colon G_n(\iota_X) \to X$ has a homotopy section $\sigma\colon X\to G_n(\iota_X)$ such that $\sigma \circ f \simeq \omega_n$.
\end{definition}

We denote this integer by $\hcat(f)$. 

Actually, speaking of  `Hopf category of $f$' is a misuse of language. We should speak of `Hopf category of the datas $\eta$, $\beta$ and $f$'.

\begin{example}
Let $X = \Sigma_A W$ and $f\simeq \id_X$. Then, as might be expected,  $\hcat(f) = 1$. Indeed, in this case, as $g_1 \circ \omega_1 \simeq f \simeq \id_X$, $\omega_1$ is a homotopy section of $g_1$. Moreover, $\omega_1 \circ f \simeq \omega_1 \circ \id_X \simeq \omega_1$, so $\hcat(f) = 1$.
\end{example}

\begin{example}\label{flechenulle}
Let $X \not\simeq *$ and $W = A \vee A$,  $\beta \simeq \pro_1 : A \vee A \to A$ and $\eta \simeq \pro_2 : A \vee A \to A$ the obvious maps. Then $\Sigma_A W \simeq *$ and we have no choice for $f$ that must be the null map $f\colon * \to X$. In this case the condition $\sigma \circ f \simeq \omega_n$ is always satisfied, so $\hcat(f) = \secat(\iota_X) = \cat(X)$.
\end{example}

Notice that $\relcat$ is a particular case of $\hcat$: When $W = A$, $\eta \simeq \beta \simeq \id_A$, then $\iota_X \simeq f$, $\omega_n \simeq \alpha_n$ and $\hcat(f) = \relcat(\iota_X)$.
Also $\relcatun$ is a particular case of $\hcat$: When $W = F_0$, then  $\Sigma_A W \simeq G_1$, $\theta \simeq \gamma_0 \simeq \alpha_1$, and if, moreover, $f \simeq g_1$, then $\omega_n \simeq \gamma_{1,n}$ and $\hcat(f) = \relcatun(\iota_X)$.

The following proposition shows that these particular cases are in fact lower and upper bounds for $\hcat(f)$.

\begin{proposition}\label{majorationhcat}
Whatever can be $f$ (and $\iota_X = f \circ \theta$), we have 
$$\secat(f) \leq \relcat(\iota_X) \leq \hcat(f) \leq \relcatun(\iota_X) \leq \relcat(\iota_X) +1.$$
\end{proposition}

\begin{proof}Consider the following homotopy commutative diagram ($n\geq 1$):
$$\xymatrix@R=1pc@C=1pc{%
&&&G_n\ar@/^/[dd]^{g_n}\\
A\ar[rr]|-{\theta}\ar@/^/[rrru]^{\alpha_n}\ar@/_/[rrrd]_{\iota_X}&&\Sigma_A W\ar[rd]^{f}\ar[ur]_{\omega_n}\\
&&&X
}$$
We see that if there is a map $\sigma\colon X \to G_n$ such that $\omega_n \simeq \sigma \circ f$ then $\alpha_n \simeq \sigma \circ \iota_X$ and this proves the second inequality.

Now consider the following homotopy commutative diagram ($n\geq 1$):
$$\xymatrix@R=1pc@C=1pc{%
&&&G_n\ar@/^/[dd]^{g_n}\\
\Sigma_A W\ar[rr]|-{\omega_1}\ar@/^/[rrru]^{\omega_n}\ar@/_/[rrrd]_{f}&&G_1\ar[rd]^{g_1}\ar[ur]_{\gamma_{1,n}}\\
&&&X
}$$
We see that if there is a map $\sigma\colon X \to G_n$ such that $\gamma_{1,n} \simeq \sigma \circ g_1$ then $\omega_n \simeq \sigma \circ f$  and this proves the third inequality.

The first inequality comes from $\secat(f)  \leq \secat(\iota_X)  \leq \relcat(\iota_X)$, the first of these two inequalities comes from \cite{DoeHa13}, Proposition 29.

Finally, the fourth inequality is proved in \cite{Doe15}.
\end{proof}

So $\hcat(f)$ establishes a  `dichotomy' between maps $f\colon \Sigma_A W \to X$: 
\begin{itemize}
\item Either $\hcat(f) = \relcat(\iota_X)$ and we have a $\sigma$ such that $f\circ \sigma \simeq \omega_n$ 
already for $n = \secat(\iota_X)$;
\item either $\hcat(f) = \relcat(\iota_X)+1$ and we have a $\sigma$ such that $f\circ \sigma \simeq \omega_n$
only for $n > \secat(\iota_X)$
\end{itemize}

\smallskip
Our last example of the section shows that the inequalities of \propref{majorationhcat} can be strict, and even that two may be strict at the same time:
\begin{example}Let $X = *$, $A \not\simeq*$ and consider $\iota_* \colon A \to *$. We have $G_i(\iota_*) \simeq A \bowtie \dots \bowtie A$, the join of $\ipu$ copies of $A$.  For any $k$, $\gamma_{k,k} \simeq \id$, so it cannot factorize through $*$; but $\gamma_{k,k+1}$ is homotopic to the null map, so $\relcatk(\iota_*) = k+1$. Now consider $f \simeq g_1(\iota_*)\colon A \bowtie A \to *$. As said before, in this case we have $\hcat(f) = \relcatun(\iota_X)$. So we get
$\secat(f) = 0 < \relcat(\iota_*) = 1 < \hcat(f) = \relcatun(\iota_*) = 2$.
\end{example}


\section{Hopf invariant and homotopy pushout}\label{cofibre}


Let us consider any homotopy commutative square:
\[\begin{split}
\xymatrix{
\Sigma_A W\ar[r]^(.6){\rho}\ar[d]_{f}&B\ar[d]^{\kappa_Y}\\
X\ar[r]_\chi&Y
}\end{split}\label{relcof}\tag{\ddag}
\]


\begin{proposition}
The homotopy commutative square above can be splitted into the following homotopy commutative diagram:
$$\xymatrix{
\Sigma_A W\ar[r]\ar[d]_{\rho}\ar@/^1pc/[rrr]^f&G_1(\iota_X)\ar[r]\ar[d]&G_n(\iota_X)\ar[d]\ar[r]&X\ar[d]^\chi\\
B\ar[r]\ar@/_1pc/[rrr]_{\kappa_Y}&G_1(\kappa_Y)\ar[r]&G_n(\kappa_Y)\ar[r]&Y
}$$
\end{proposition}

\begin{proof}Set $\phi = \rho\circ\theta$. Since $\theta\circ\eta \simeq \theta\circ\beta$, also $\phi\circ\eta \simeq \phi\circ\beta$. 
First notice that we can insert the original homotopy square inside the following homotopy commutative diagram:
\[\xymatrix@R=1pc@C=1pc{
&W\ar[rr]^\eta\ar[ld]_\beta\ar[dd]|\hole&&A\ar[ld]\ar[dd]^(.7)\phi|\hole\ar[rd]^{\iota_X}\\
A\ar[rr]\ar[dd]_\phi&&\Sigma_A W\ar[dd]^(.3){\rho}\ar[rr]_(.4)f&&X\ar[dd]^{\chi}\\
&B\ar@{=}[dl]\ar@{=}[rr]|!{[ur];[dr]}\hole && B\ar@{=}[dl]\ar[dr]\\
B\ar@{=}[rr]&&B\ar[rr]_{\kappa_Y}&&Y
}\]

By induction on $n \geq 1$, starting from the outside cube of the above diagram and $\phi_0 = \phi$, we can build a homotopy diagram:
\[\begin{split}\xymatrix@R=1pc@C=1pc{
W\ar[dd]\ar[rr]\ar[rd]&&F_{n-1}(\iota_X)\ar[rr]\ar[ld]\ar@{-->}[dd]|\hole&&A\ar[ld]\ar[dd]^(.7)\phi|\hole\ar[rd]^{\iota_X}\\
&G_{n-1}(\iota_X)\ar[rr]\ar[dd]^(0.3){\phi_{n-1}}&&G_n(\iota_X)\ar@{.>}[dd]^(0.3){\phi_n}\ar[rr]&&X\ar[dd]^{\chi}\\
B \ar[rd]\ar[rr]|(0.45)\hole&&F_{n-1}(\kappa_Y)\ar[dl]\ar[rr]|!{[ur];[dr]}\hole&& B\ar[dl]\ar[dr]_(.45){\kappa_Y}\\
&G_{n-1}(\kappa_Y)\ar[rr]&&G_n(\kappa_Y)\ar[rr]_{g_n}&&Y
}\end{split}
\]
where the dashed and dotted maps are induced by the homotopy pullback $F_{n-1}(\kappa_Y)$ and the homotopy pushout $G_n(\iota_X)$ respectively.

So we obtain a homotopy commutative diagram:
\[\xymatrix@R=1pc@C=1pc{
&W\ar[rr]\ar[ld]\ar[dd]|\hole&&A\ar[ld]\ar[dd]|\hole\ar[rd]^{\iota_X}\\
A\ar[rr]\ar[dd]&&G_n(\iota_X)\ar@{.>}[dd]\ar[rr]_(.4){g_n}&&X\ar[dd]^{\chi}\\
&B\ar@{=}[dl]\ar@{=}[rr]|!{[ur];[dr]}\hole && B\ar[dl]\ar[dr]\\
B\ar[rr]&&G_n(\kappa_Y)\ar[rr]_{g_n}&&Y
}\]

Finally take the homotopy pushout inside the upper and lower lefter squares to get the homotopy commutative diagram:
$$\xymatrix{
\Sigma_A W\ar[d]_{\rho}\ar[r]_{\omega_n}&G_n(\iota_X)\ar[d]\ar[r]&X\ar[d]^\chi\\
B\ar[r]&G_n(\kappa_Y)\ar[r]&Y
}$$
and this gives the required splitting of the original square.
\end{proof}


\begin{proposition}\label{relcatstable} If the square \ref{relcof} is a homotopy pushout, then
 $$\relcat(\kappa_Y) \leq \hcat(f).$$
\end{proposition}

As a particular case, when $B \simeq *$, $Y$ is the homotopy cofibre of $f$, and $\relcat(\kappa_Y) = \cat(Y)$. So the Proposition asserts that $\hcat(f) \geq \cat(Y)$.

\begin{proof}Let $\hcat(f) \leq n$,  so we have a homotopy section $\sigma$ of $g_n(\iota_X)$ such that $\sigma\circ f \simeq \omega_n$.
First apply the `Whisker maps inside a cube' lemma to the outer part of the following homotopy commutative diagram:
$$\xymatrix@R=1pc@C=1pc{
&\Sigma_A W\ar[ld]_{\omega_n}\ar[rr]\ar@{=}[dd]|!{[ld];[rd]}\hole&&B\ar[ld]_a\ar[rd]^{\alpha_n}\ar@{=}[dd]|\hole&\\
G_n(\iota_X)\ar[rr]\ar[dd]&&S\ar@{.>}[dd]_(.3){c}\ar@{.>}[rr]_(.25)b&&G_n(\kappa_Y)\ar[dd]^{g_n}\\
&\Sigma_A W\ar[rr]|!{[ld];[rd]}\hole\ar[ld]&&B\ar[ld]\ar[rd]&\\
X\ar[rr]&&Y\ar@{=}[rr]&&Y
}$$
where the inner horizontal squares are homotopy pushouts, and $c$ and $b$ are the whisker maps induced by the homotopy pushout $S$. 
Next build the following homotopy commutative diagram:
$$\xymatrix@R=1pc@C=1pc{
&\Sigma_A W\ar[ld]_{f}\ar[rr]\ar@{=}[dd]|!{[ld];[rd]}\hole&&B\ar[ld]_(.6){\kappa_Y}\ar@{=}[dd]&\\
X\ar[rr]\ar[dd]_\sigma&&Y\ar@{.>}[dd]_(0.3)d\\
&\Sigma_A W\ar[rr]|!{[ld];[rd]}\hole\ar[ld]_{\omega_n}&&B\ar[ld]^a\ar[rd]^{\alpha_n}&\\
G_n(\iota_X)\ar[rr]&&S\ar[rr]_b&&G_n(\kappa_Y)
}$$
where $d$ is  the whisker map induced by the homotopy pushout $Y$.
Let $\sigma' = b\circ d$. We have $g_n \circ \sigma' \simeq g_n \circ b \circ d \simeq c \circ d \simeq \id_C$ and
$\sigma' \circ \kappa_Y \simeq b\circ d\circ \kappa_Y \simeq b\circ a \simeq \alpha_n$.
\end{proof}

\begin{corollary}\label{relcatinstable}In the diagram \ref{relcof}, if $\relcat(\kappa_Y) = \relcat(\iota_X) +1$, then $\hcat(f) = \relcat(\iota_X) + 1$.
\end{corollary}

\begin{proof}By \propref{relcatstable}, the hypothesis implies that $\hcat(f) \geq \relcat(\iota_X) + 1$. But by  \propref{majorationhcat}, we have $\hcat(f) \leq \relcat(\iota_X) + 1$. So we have the equality.
\end{proof}

It is now easy to exhibit examples of maps $f$ with $\hcat(f) = \relcat(\iota_X) + 1$. Indeed there are plenty examples of homotopy pushouts where $\relcat(\kappa_Y) = \relcat(\iota_X) + 1$:

\begin{example}Let $A = B = *$ and $f\colon S^r\to S^n$ be any of the Hopf maps $S^3 \to S^2$, $S^7 \to S^4$ or $S^{15}\to S^8$. 
So here $\relcat(\iota_X) = \cat(S^n) = 1$. On the other hand it is well known that those maps have a homotopy cofibre $S^n/S^r$ of category 2, so here $\relcat(\kappa_Y) = \cat(S^n/S^r) =2$. By \cororef{relcatinstable}, we have $\hcat(f) = 2$.
\end{example}

\begin{example}Let $f$ be the map $u$ in the homotopy cofibration 
$$\xymatrix{Z \bowtie Z\ar[r]_(.45)u& \Sigma Z\vee \Sigma Z \ar[r]_{t_1}& \Sigma Z \times \Sigma Z}$$
where $Z \bowtie Z \simeq \Sigma (Z \wedge Z)$ is the join of two copies of $Z$
and is also the suspension of the smash product of two copies of $Z$. 
Let $A = B = *$, $\Sigma Z \not\simeq *$.
We have $\relcat(\iota_X) = \cat(\Sigma Z\vee \Sigma Z) = 1$ and $\relcat(\kappa_Y) = \cat(\Sigma Z\times \Sigma Z) = 2$, so by \cororef{relcatinstable} again, we have $\hcat(u) = 2$.
\end{example}

\begin{example}
For  $i \geq 1$, let $f$ be the map $\beta_i$ in the Ganea construction:
$$\xymatrix{
A\ar[r]^{\theta_i}\ar[rd]_{\alpha_i}&F_i\ar[r]^{\eta_i}\ar[d]^{\beta_i}&A\ar[d]^{\alpha_{i+1}}\\
&G_i\ar[r]_{\gamma_i}&G_{i+1} 
}$$

Actually $F_i$ is a join over $A$ of $\ipu$ copies of $F_0$, and also a relative suspension $\Sigma_A W$ where $W$ is a relative smash product. 
For any $i \leq \relcat(\iota_X)$, we have $\relcat(\alpha_i) = i$, see \cite{DoeHa13}, Proposition 23.
So by \cororef{relcatinstable} again, if $i < \relcat(\iota_X)$, we have $\hcat(\beta_i)  = \relcat(\alpha_i) + 1 = i+1$.
\end{example}

\section{The Strong Hopf category}\label{strong}

In \cite{DoeHa13}, we introduced the strong version of $\relcat$, namely $\Relcat$.
In this section, we introduce the strong version of $\hcat$, namely $\Hcat$. This gives an alternative way, sometimes usefull, to see if a map has a Hopf category less or equal to $n$. Also this will lead to a new inequality: $\hcat(f) \geq \relcat(f)$. Consequently, if $\relcat(f) > \relcat(\iota_X)$, then $\hcat(f) = \relcat(\iota_X) + 1$.

\begin{definition}The \emph{strong Hopf category of a  map $f\colon \Sigma_AW \to X$} is the least integer $n \geq 1$ such that:
\begin{itemize}
\item there are maps $\iota_0\colon A \to X_0$ and a homotopy inverse $\lambda\colon X_0\to A$, i.e.  $\iota_0
  \circ \lambda \simeq \id_{X_0}$ and $\lambda \circ \iota_0 \simeq \id_A$; 
\item for each $i$, $0\leq i < n$, there is a homotopy commutative cube:
\[
\begin{split}
\xymatrix@R=1pc@C=1pc{
&W\ar[rr]^\beta\ar[ld]_\eta\ar[dd]|\hole&&A\ar[ld]\ar[dd]^{\iota_i}\\
A\ar[rr]\ar@{=}[dd]&&\Sigma_AW\ar[dd]^(.3){\zeta_{i+1}}&\\
&Z_i\ar[dl]\ar[rr]^(.3){z_i}|!{[ur];[dr]}\hole&&X_i\ar[dl]^{\chi_i}\\
A\ar[rr]_{\iota_{i+1}}&&X_{i+1}
}\end{split}\label{Hcat}\tag{$\natural$}
\]
where the bottom square is a homotopy pushout.
\item $X_n=X$ and   $\zeta_n \simeq f$.
\end{itemize}
\end{definition}

We denote this integer by $\Hcat(f)$.

\smallskip
Notice that $\iota_{i+1} \simeq \zeta_{i+1} \circ \theta \simeq \chi_i \circ \iota_i$. In particular, this means that $\Pushcat(\iota_i) \leq i$ in the sense of \cite{DoeHa13}, Definition 3.

\smallskip
For $0 \leq i \leq n$, define the sequence of maps $\xi_i \colon X_i \to X$  with the relation $\xi_{i} = \xi_{i+1} \circ \chi_{i}$ (when $i < n$), starting with $\xi_n = \id_X$. 
We have $\xi_n \circ \iota_n \simeq \iota_X$ and $\xi_{i}\circ \iota_{i} = \xi_{i+1} \circ \chi_{i} \circ \iota_{i} \simeq \xi_{i+1} \circ
\iota_{i+1} \simeq \iota_X$ by decreasing induction. Also $\iota_X \circ \lambda \simeq \xi_0 \circ \iota_0 \circ \lambda \simeq \xi_0$.
Moreover, for $0 < i \leq n$ we have we have $\xi_i \circ \zeta_i \simeq f$ by the `Whisker maps inside a cube lemma'. 
So we have the following homotopy diagram:
$$\xymatrix@R=1pc@C=1pc{
&W\ar[rr]^\eta\ar[ld]_\beta\ar[dd]|\hole&&A\ar[ld]\ar@{=}[dd]|\hole\ar[rd]^\theta\\
A\ar[rr]\ar[dd]_{\iota_i}&&\Sigma_AW\ar[dd]^(.3){\zeta_{i+1}}\ar@{=}[rr]&&\Sigma_AW\ar[dd]^f\\
&Z_i\ar[dl]\ar[rr]|!{[ur];[dr]}\hole&&A\ar[dl]^{\iota_{i+1}}\ar[rd]^(.4){\iota_X}\\
X_i\ar[rr]_{\chi_{i}}&&X_{i+1}\ar[rr]_{\xi_{i+1}}&&X
}$$

\smallskip
We say that a map $g\colon B \to Y$ is `relatively dominated' by a map $f\colon B \to X$ if there is a map $\varphi\colon X\to Y$ with a homotopy section $\sigma\colon Y \to X$ such that $\varphi\circ f \simeq g$ and $\sigma \circ g \simeq f$.

\begin{proposition}
A map $g\colon \Sigma_AW \to Y$ has $\hcat(g) \leq n$ iff $g$ est relatively dominated by a map $f \colon \Sigma_AW \to X$ with $\Hcat(f) \leq n$.
\end{proposition}
\begin{proof}
Consider the map $\omega_n\colon \Sigma_A W \to G_n(\iota_Y)$ as in diagram \ref{diag1} and notice that $\Hcat(\omega_n) \leq n$.
If  $\hcat(f)\leq n$, then $f$ is relatively dominated by $\omega_n$.

\smallskip
For the reverse direction, by hypothesis, we have a map $\varphi$ and a homotopy section $\sigma$  such that
$\varphi\circ f \simeq g$ and $\sigma \circ g \simeq f$; composing with $\theta$, we have also $\varphi \circ \iota_X \simeq \iota_Y$ and $\sigma \circ \iota_Y \simeq \iota_X$.
 From the hypothesis $\Hcat(f)\leq n$, we get a sequence of homotopy commutative  diagrams, for $0\leq i < n$, 
which gives the top part of the following diagram.
 
We show by induction that the map
$\varphi\circ \xi_i\colon X_i\to Y$ factors through $g_i\colon G_i(\iota_Y)\to Y$ up to homotopy.
This is true for $i=0$ since we have $\xi_0 \simeq \iota_X \circ \lambda$,
so $\varphi\circ \xi_0 \simeq \varphi\circ\iota_X \circ \lambda \simeq \iota_Y
\circ \lambda = g_0 \circ \lambda$.
Suppose now that we have a map $\lambda_{i}\colon X_{i}\to G_{i}(\iota_Y)$
such that $g_{i}\circ \lambda_{i}\simeq \varphi\circ \xi_{i}$. Then we construct a homotopy commutative diagram
$$\xymatrix@C=1pc@R=1pc{
&Z_{i}\ar[ld]_{z_i}\ar[rr]\ar@{-->}[dd]|!{[ld];[rd]}\hole
&&A\ar[ld]_{\iota_{i+1}}\ar[rd]\ar@{=}[dd]|\hole&\\
X_{i}\ar[rr]\ar[dd]_{\lambda_{i}}&&
X_{i+1}\ar@{..>}[dd]_(.3){\lambda_{i+1}}\ar[rr]_(.34){\xi_{i+1}}&&X\ar[dd]^{\varphi}\\
&F_{i}\ar[rr]|(.5)\hole   
\ar[ld]&&A\ar[ld]^{\alpha_{i+1}}\ar[rd]&\\
G_{i}(\iota_Y)\ar[rr]&&G_{i+1}(\iota_Y)\ar[rr]_{g_{i+1}}&&Y
}$$
where $\xymatrix@1{Z_{i}\ar@{-->}[r]& F_{i}}$ is  the whisker map induced by the bottom homotopy pullback and 
$\lambda_{i+1}\colon \xymatrix@1{X_{i+1}\ar@{.>}[r]&  G_{i+1}(\iota_Y)}$ is the whisker map induced by the top homotopy pushout.
The composite $g_{i+1}\circ \lambda_{i+1}$ is homotopic to $\varphi\circ \xi_{i+1}$.
Hence the inductive step is proven.

At the end of the induction, we have $g_n\circ \lambda_n\simeq \varphi\circ \xi_n = \varphi \circ \id_X = \varphi$. As we have a homotopy section $\sigma\colon Y\to X_n=X$ of $\varphi$, we get a homotopy section $\lambda_n\circ \sigma$ of $g_n$.
Moreover, we have $(\lambda_n\circ \sigma) \circ g \simeq \lambda_n \circ f \simeq \lambda_n \circ \zeta_n \simeq \omega_n$. 
\end{proof}

\begin{example}If we consider any relative suspension $\Sigma_A f \colon \Sigma_A W \to \Sigma_A Z$ (and in particular, of course, when $A = *$, any suspension $\Sigma f \colon \Sigma W \to \Sigma Z$), we have $\Hcat(\Sigma_A f) = 1$. And so, any map g that is relatively dominated by a (relative) suspension has $\hcat(g) = 1$.
\end{example}

In fact, by definition, a map $g$ has $\Hcat(g) = 1$ if and only if $g$ is a (relative) suspension. There are maps for which the  strong Hopf category is greater than the Hopf category: For instance, consider the null map $f\colon * \to X$ of \examref{flechenulle}; if $X$ is a space with $\cat(X) = 1$ that is not a suspension, then $f$ cannot be a suspension, so $\Hcat(f) > \hcat(f) = 1$.

\begin{proposition}\label{Hcatzeta}
In the diagram \ref{Hcat}, we have $$\Relcat(\zeta_i) \leq i$$
\end{proposition}

As $\omega_i$ is a particular case of $\zeta_i$, this implies $\Relcat(\omega_i) \leq i$.

\begin{proof}For $i > 0$, let build the following homotopy diagram where the three squares are homotopy pushouts:
\[
\xymatrix{
W \ar[r]\ar[d]_\beta\ar@/^{1pc}/[rr]^\eta &Z_\imu \ar[r]\ar[d]\ar@/^{1.5pc}/[dd]^(0.3){z_\imu} &A\ar[d]_\theta\ar@/^{1.5pc}/[dd]^{\iota_i}\\
A \ar[r]\ar@/_/[rd]_{\iota_\imu} &C_{i-1} \ar[r]\ar[d]_{c_\imu}&\Sigma_A W\ar[d]_{\zeta_i}\\
&X_\imu\ar[r]&X_i
}
\]
 and where the map $c_{i-1} = (\iota_{i-1},z_{i-1})$ is the whisker map induced by the homotopy pushout.

We have $\secat(\iota_\imu) \leq \Pushcat(\iota_\imu) \leq \imu$ by \cite{DoeHa13}, Theorem 18. 
So $\secat(c_\imu) \leq \imu$ by \cite{DoeHa13}, Proposition 29.
So $\Relcat(c_\imu) \leq (i-1) + 1 = i$ by \cite{DoeHa13}, Theorem 18.
And this implies $\Relcat(\zeta_i) \leq i$ by \cite{DoeHa13}, Lemma 11.
\end{proof}

\begin{theorem}\label{relcatpphcat}
For any $f\colon \Sigma_A W\to X$, we have
$$\Relcat(f) \leq \Hcat(f) \quad \mbox{and} \quad \relcat(f) \leq \hcat(f)$$
\end{theorem}

\begin{proof}If $\Hcat(f) = n$, then we have $f\simeq \zeta_n$ in \ref{Hcat}.  So $\Relcat(f) = \Relcat(\zeta_n) \leq n$ by \propref{Hcatzeta}.

If $\hcat(f) = n$, then $f$ is relatively dominated by $\omega_n$. As $\Relcat(\omega_n) \leq n$, we have $\relcat(f) \leq  n$ by \cite{DoeHa13}, Proposition 10.
\end{proof}

As a corollary, we get an indirect proof of \propref{relcatstable} because $\relcat(\kappa_Y) \leq \relcat(f)$ by \cite{DoeHa13}, Lemma 11, that asserts that a homotopy pushout doesn't increase the relative category.

\smallskip
It is not difficult to find an example where these inequalities are strict:

\begin{example}\label{t1}Let $f$ be the map $t_1$ in the homotopy cofibration 
$$\xymatrix{Z \bowtie Z\ar[r]_(.45)u& \Sigma Z\vee \Sigma Z \ar[r]_{t_1}& \Sigma Z \times \Sigma Z}$$
Let $A = *$, $\Sigma Z \not\simeq *$.
As $t_1$ is a homotopy cofibre, we have $\relcat(t_1) \leq \Relcat(t_1) \leq 1$, see \cite{DoeHa13}, Proposition 9.
On the other hand, we have $\Hcat(t_1) \geq \hcat(t_1) \geq \relcat(\iota_X) = \cat(\Sigma Z\times \Sigma Z) = 2$ by \propref{majorationhcat}.
\end{example}

\section{Equivalent conditions to get the Hopf category}\label{sigmabars}

Let be given any map $f\colon \Sigma_A W \to X$ with $\secat(\iota_X) \leq n$ and any homotopy section $\sigma\colon X \to G_n$ of $g_n\colon G_n\to X$.  Consider the following homotopy pullbacks:
$$\xymatrix{
Q\ar[d]_{\pi'}\ar[r]^\pi&\Sigma_A W\ar[d]_{\theta_n^W}\ar@{=}[rd]\\
\Sigma_A W\ar[r]_{\bar \sigma}\ar[d]_{f}&H_n\ar[r]_{\eta_n^W}\ar[d]_{f_n}&\Sigma_A W\ar[d]^{f}\\
X\ar[r]_\sigma&G_n\ar[r]_{g_n}&X
}$$
where  $\theta_n^W = (\omega_n, \id_{\Sigma_A X})$ is the whisker map induced by the homotopy pullback $H_n$. By the `Prism lemma' (see \cite{DoeHa13}, Lemma 46, for instance ), we know that the homotopy pullback of $\sigma$ and $f_n$ is indeed $\Sigma_A W$, and that $\eta_n^W \circ \bar\sigma \simeq \id_{\Sigma_A W}$.
Also notice that $\pi \simeq \pi'$ since $\pi \simeq \eta_n^W \circ \theta_n^W \circ \pi \simeq \eta_n^W \circ \bar\sigma \circ \pi' \simeq \pi'$.

\begin{proposition}\label{sigmabar}Let be given any map $f\colon \Sigma_A W \to X$ with $\secat(\iota_X) \leq n$ and any homotopy section $\sigma\colon X \to G_n(\iota_X)$ of $g_n\colon G_n(\iota_X)\to X$. With the same definitions and notations as above, the following conditions are equivalent:
\begin{enumerate}
\item $\sigma \circ f \simeq \omega_n$.
\item $\pi$ has a homotopy section.
\item $\pi$ is a homotopy epimorphism.
\item $\theta_n^W \simeq \bar\sigma$.
\end{enumerate}
\end{proposition}

\begin{proof}We have the following sequence of implications:

(i) $\implies$ (ii):  Since $\sigma \circ f \simeq  \omega_n  \simeq f_n \circ \theta_n^W \circ \id_{\Sigma_A W}$, we have a whisker map $(f,\id_{\Sigma_A W})\colon {\Sigma_A W}\to Q$ induced by the homotopy pullback $Q$ which is a homotopy section of $\pi$.

(ii) $\implies$ (iii): Obvious.

(iii) $\implies$ (iv): We have $\theta_n^W \circ \pi \simeq  \bar\sigma \circ \pi$ since $\pi \simeq \pi'$. Thus  $\theta_n^W \simeq \bar\sigma$ since $\pi$ is a homotopy epimorphism.

(iv) $\implies$ (i): We have $\sigma \circ f \simeq f_n \circ  \bar\sigma \simeq f_n \circ \theta_n^W \simeq \omega_n$.
\end{proof}

\begin{theorem}\label{dimension}Let be   a $(q-1)$-connected map $\iota_X\colon A\to X$ with $\secat \iota_X \leq n$. If $\Sigma_A W$ is a CW-complex with $\dim \Sigma_A W < (n + 1)q-1$ then $\sigma \circ f \simeq \omega_n$ for any homotopy section $\sigma$ of $g_n$.\end{theorem}

\begin{proof}Recall that $g_i$ is the $(i+1)$-fold join of $\iota_X$. Thus by \cite{Mat76}, Theorem 47, we obtain that, for each $i \geq 0$, $g_i : G_i \to X$ is $(i+1)q-1$-connected. As $g_i$ and $\eta_i^W$ have the same homotopy fibre, the Five lemma implies that $\eta_i^W\colon H_i \to \Sigma_A W$ is $(i+1)q-1$-connected, too. By \cite{Whi78}, Theorem IV.7.16, this means that for every CW-complex $K$ with $\dim K < (i+1)q-1$, $\eta_i^W$ induces a one-to-one correspondence $[K,H_i] \to [K,\Sigma_A W]$. Apply this to $K = \Sigma_A W$ and $i=n$: Since $\theta_n^W$ and $\bar\sigma$ are both homotopy sections of $\eta_n^W$, we obtain $\theta_n^W \simeq \bar\sigma$, and \propref{sigmabar} implies the desired result.
\end{proof}

\begin{example}Let $A = \ast$ and $W = S^{r-1}$, so $\Sigma_A W = S^r$, and $X = S^m$. In this case $\secat \iota_X = \cat S^m = 1$. Hence \thmref{dimension} means that if $r < 2m-1$, we have $\sigma \circ f \simeq \omega_1$, whatever can be $f$ and $\sigma \colon X \to G_1(\iota_X)$, so $\hcat f =1$ and  we get by \propref{relcatstable} that the homotopy cofibre $C$ of $f$ has $\cat C \leq 1$. (Notice that if $r < m$ then $f$ is a nullhomotopic, so $C$ is simply $S^m \vee S^{r+1}$.) 
\end{example}

\begin{example}Let  $A =*$, $\Sigma W \simeq \Sigma(S^{r-1} \vee S^{r-1}) \simeq S^r \vee S^r$, $X \simeq  S^r \times S^r$ and consider $t_1\colon S^r \vee S^r \to S^r \times S^r$. Here $\secat(\iota_X) = \cat(S^r \times S^r) = 2$. For any $r \geq 1$, we have $\dim (S^r \vee S^r) = r < (2+1)r - 1$, so $\hcat(t_1)=2$.
\end{example}

\bibliographystyle{plain}
\bibliography{secat}

\end{document}